\documentclass{compositio}

\title{Extensions of differential representations of $\SL_2$ and  tori}
\author{Andrey Minchenko}
\email{aminchen@uwo.ca}
\address{University of Western Ontario\\ Department of Mathematics\\
 London, ON N6A 5B7, Canada}

\author{Alexey Ovchinnikov}
\email{aovchinnikov@qc.cuny.edu}
\address{City University of New York,
Queens College, Department of Mathematics,
65-30 Kissena Blvd, Flushing, NY 11367, USA
}

\thanks{A. Ovchinnikov was supported by the grants: NSF  CCF-0952591 and  PSC-CUNY  No.~60001-40~41.}

\date\today

\usepackage[pdfstartview=FitH,
            colorlinks,
           linkcolor=reference,
            citecolor=citation,
            urlcolor=e-mail,
           backref]{hyperref}

\usepackage{rnote,amsmath,theorems,leftidx}
\usepackage{tikz}

\def\SL{{\bf SL}}
\DeclareMathOperator{\Z}{\mathbb Z}
\DeclareMathOperator{\Ker}{\mathrm Ker}
\DeclareMathOperator{\wt}{\mathrm wt}

\DeclareMathOperator{\U}{\mathcal U}

\DeclareMathOperator{\Id}{\mathrm Id}

\DeclareMathOperator{\I}{\mathbb I}
\DeclareMathOperator{\Q}{\mathbb Q}

\DeclareMathOperator{\Const}{\mathcal C}

\DeclareMathOperator{\Ob}{Ob}
\DeclareMathOperator{\Char}{char}

\newcommand{\Le}{\leqslant}
\newcommand{\Ge}{\geqslant}

 \usepackage{mathptmx}



\begin{abstract} 
Linear differential algebraic groups (LDAGs) measure differential algebraic dependencies among solutions of linear differential and difference equations with parameters, for which LDAGs are Galois groups.
The differential representation theory is a key to developing algorithms computing these groups. 
 In the rational representation theory of algebraic groups, one starts with $\SL_2$ and tori to develop the rest of the theory. 
 In this paper,
we give an explicit description of differential representations of tori and differential extensions of irreducible representation of $\SL_2$. 
In these extensions, the two irreducible representations can be non-isomorphic. This is in contrast to differential representations of tori, which turn out to be direct sums of isotypic representations.
\end{abstract}

\begin{document}
\maketitle

\section{Introduction}
Linear differential algebraic groups (LDAGs) were introduced in \cite{Cassidy,CassidyRep,KolDAG,CassidyClassification} and
are now extensively used to study ordinary and partial differential and difference equations~\cite{CassidySinger,CharlotteLucia,CharlotteComp,hardouin_differential_2008,MichaelPhyllisFactoring,OvchRecoverGroup,OvchTannakian}, where these groups play the role of Galois groups and measure differential algebraic dependencies among the solutions.
Due to~\cite{SitSL2}, one has a complete description of differential algebraic subgroups of the LDAG $\SL_2$. However, in order to develop algorithms for the differential and difference equations mentioned above, knowledge of the differential representation theory is essential. But, even the differential representation theory of $\SL_2$ is largely unknown, with the initial observations made in~\cite{diffreductive}. In the present paper, we make a first step in resolving this problem. 

Our main result, Theorem~\ref{thm:main}, is an explicit description of differential extensions of irreducible representations of $\SL_2$ over an ordinary differential field $\K$ of characteristic zero\footnote{Although we consider the case of one derivation on $\K$, it is possible to carry out our constructions in the case of several commuting derivations. However, this would significantly increase the complexity of notation without introducing new ideas. Hence, we have decided not to include this case into the present paper.}, not necessarily differentially closed. 
However, we require that $\K$ has an element whose derivative is not zero. The main idea is to construct an embedding of such a representation or its dual into the ring $\K\{x,y\}$ of differential polynomials in two differential indeterminates. However, if a differential representation of $\SL_2$ is an extension of more than two irreducible representations, it might not be embeddable into $\K\{x,y\}$ as Example~\ref{ex:counterexample} shows. This demonstrates  one of the numerous subtleties that differential representations have.

In the classical rational representation theory of the algebraic group $\SL_2$ in characteristic zero, every finite-dimensional $\SL_2$-module is a direct sum of irreducible ones, and each of those is isomorphic to
$$
\Span_\K\left\{x^d,x^{d-1}y,\ldots, xy^{d-1},y^d\right\} \subset \K[x,y],
$$
for some $d \Ge 0$, where the action of $\SL_2$ is:
$$
\SL_2 \ni \begin{pmatrix}
a& b\\
c&d
\end{pmatrix} \mapsto
\begin{cases}
x\mapsto ax+cy,\\
y \mapsto bx + dy.
\end{cases}
$$
However, this approach does not directly generalize to differential representations of $\SL_2$ for various reasons. On the one hand, the irreducible ones are all algebraic (given by polynomials without derivatives) \cite[Theorem~3.3]{diffreductive} and, therefore, are fully described as above. On the other hand, not every differential representation of $\SL_2$ is a direct sum of irreducible ones \cite[Theorem~3.13, Example~3.16 and Remark~4.9]{diffreductive}.  Hence, to describe them, we will need to characterize all indecomposable differential representations, that is, the ones that are not direct sums of any proper subrepresentations. All other differential representations will, therefore, be direct sums of those.

In order to follow this different approach, we first obtain all indecomposable representations from the ones that have only one minimal and one maximal subrepresentation using standard pull-backs and push-outs (Sections~\ref{sec:Rep0} and~\ref{sec:ss}). Now, it only remains to characterize that special subclass, denoted by $\Rep_0$, of indecomposable representations. The goal is to produce a description that is easy to use. For this, we first embed every representation of $\Rep_0$ into the ring of differential polynomial functions on $\SL_2$, which is the quotient of the ring of differential polynomials in four differential indeterminates by the differential ideal generated by  $\det-1$ (see~\eqref{eq:ABP}, Proposition~\ref{prop:reg}, and Example~\ref{ex:reg}). However, the presence of  differential relations in the quotient makes it difficult to use. 

Certainly, to embed representations  from $\Rep_0$ (or their duals, at least) into $\K\{x,y\}$ would be desirable but is impossible as we have already pointed out (Example~\ref{ex:counterexample}). However, we discover an important subset of $\Rep_0$ for which it is true that each representation or its dual embed into $\K\{x,y\}$. These representations are extensions of two irreducible $\SL_2$-modules, and are the main ingredients of our paper. 
Moreover, after embedding the representation into $\K\{x,y\}$, we show how to characterize these extensions inside $\K\{x,y\}$. This is the only place where we use the requirement for $\K$ to contain a non-constant element (see Lemma~\ref{lem:firstorder} as well as the preparatory results from Section~\ref{sec:Gmdiffpoly}).

The situation is very much different for differential representations of tori (whose differential algebraic subgroups were characterized in \cite[Chapter~IV]{Cassidy}). In particular, as we show for comparison in Theorem~\ref{thm:tori}, the only indecomposable differential representaions of a torus are extensions of isomorphic irreducible representations. This is certainly much simpler to handle than even our case of differential extensions of two irreducible representations of $\SL_2$, showing another subtlety that we have to face and deal with here.

One can apply the differential representation theory of $\SL_2$ to developing an algorithm that computes the differential Galois group of a system of linear differential equations with parameters. Such an algorithm for the non-parameterized Galois theory  usually operates with a list of groups that can possibly occur and step-by-step eliminates the choices~\cite{Kovacic1,FelixMichael1,FelixMichael2,FelixMichael3,HRUW,FelixJAW,van_der_put_galois_2003}. In the parameterized case, one can determine  the possible  block structures (by factoring the original differential equation and its prolongations with respect to the parameter if needed): the sizes of the irreducible diagonal blocks and whether the extensions they form are trivial. It turns out that this and the classification results from our paper combined with the reductivity test that is being developed in~\cite{CharlotteAlexey} are definitive enough for the ``elimination process'' mentioned above to become a part of an algorithm for parameterized systems of order up to $4$, see also~\cite{Dreyfus}.

The paper is organized as follows. We recall the basic definitions of differential algebra and differential algebraic groups in Section~\ref{sec:basicdef}. In Section~\ref{sec:preparation}, we also recall how to construct all representation
from our building blocks, representations with one minimal and maximal subrepresentations. Section~\ref{sec:main}, the main part of the paper, starts with a description of all indecomposable differential representations of tori in Section~\ref{sec:Gm}, which we then compare with differential representations of $\SL_2$ in Section~\ref{sec:SL2} and show our main result, Theorem~\ref{thm:main}, there. We finish the paper by an example demonstrating that the hypothesis of our main result cannot be relaxed. 

\section{Basic definitions}\label{sec:basicdef}
 A $\partial$-ring $R$ is a commutative associative ring with unit $1$ and a  derivation $\partial: R\to R$ such that
$$
\partial(a+b) = \partial(a)+\partial(b),\quad \partial(ab) =
\partial(a)b + a\partial(b)
$$
for all $a, b \in R$. 
For example, $\Q$ is a $\partial$-field (a field and a $\partial$-ring at the same time) with the unique
possible derivation (which is the zero one). The field
$\mathbb{C}(t)$ is also a $\partial$-field with $\partial(t) = f,$
and this $f$ can be any element of $\mathbb{C}(t).$
Let
$$
\Theta = \left\{\partial^i\:|\: i \Ge 0\right\}.$$
Since $\partial$ acts on $R$,
there is a natural action of $\Theta$ on $R$. For $r \in R$, we also denote $\partial r$ by $r'$  and $\partial^i r$ by $r^{(i)}$, $i \Ge 2$, whenever it is convenient.

Let $R$ be a $\partial$-ring. If $B$ is an $R$-algebra, then $B$ is a $\partial$-$R$-algebra
if the action of $\partial$ on $B$ extends the
action of $\partial$ on $R$.
Let $Y = \{y_1,\ldots,y_n\}$ be a set of variables. We differentiate them:
$$
\Theta Y := \left\{\partial^iy_j
\:\big|\: i \Ge 0,\ 1\Le j\Le n\right\}.
$$
The ring of differential polynomials $R\{Y\}$ in
differential indeterminates $Y$
over $R$ is
the ring of commutative polynomials $R[\Theta Y]$
in infinitely many algebraically independent variables $\Theta Y$ with
the derivation $\partial$ that 
extends the $\partial$-action on $R$ as follows:
$$
\partial\left(\partial^i y_j\right) := \partial^{i+1}y_j,\quad i \Ge 0,\ 1 \Le j \Le n.$$
 An ideal $I$ in a $\partial$-ring $R$ is called differential if it is stable under the action of
$\partial$, that is,
$
\partial(a) \in I$ for all  $a \in I$. 
 If $F \subset R$, then $[F]$ denotes the differential ideal generated by $F$.

We shall recall some definitions and results from differential
algebra (see  \cite{Cassidy,Kol} for more detailed information)
leading up to the ``classical definition'' of a linear differential
algebraic group.
Let $\K$ be a $\partial$-field. In what follows, we will assume that $\Char\K = 0$.
Let $\U$ be a  differentially closed field containing $\K$ (see \cite[Definition 3.2]{CassidySinger}, \cite[Definition~4]{TrushinSplitting}, and the references given there).
Let also $\Const\subset\U$ be its subfield of constants\footnote{One can show that the field $\Const$ is algebraically closed.}, that is, $\Const = \ker\partial$.

\begin{definition} For a differential field extension $K\supset \K$, a {\it Kolchin closed} subset $W(K)$ of $K^n$ over $\K$ is the set of common zeroes
of a system of differential algebraic equations with coefficients in $\K,$ that is, for $f_1,\ldots,f_k \in \K\{Y\}$ we define
$$
W(K) = \left\{ a \in K^n\:|\: f_1(a)=\ldots=f_k(a) = 0\right\}.$$
\end{definition}

There is a bijective correspondence between Kolchin closed subsets
$W$ of $\U^n$ defined over $\K$ and radical differential ideals
$\I(W) \subset \K\{y_1,\ldots,y_n\}$
generated by the differential polynomials $f_1,\ldots,f_k$ that define $W$.
In fact, the $\partial$-ring $\K\{Y\}$ is
Ritt-Noetherian, meaning that every radical
differential ideal is the radical of a finitely
generated differential ideal by the Ritt-Raudenbush basis theorem.
Given a Kolchin closed subset $W$ of
$\U^n$ defined over $\K$, we let the {\it coordinate ring} $\K\{W\}$
be
$$
\K\{W\} = \K\{y_1,\ldots,y_n\}\big/\I(W).
$$
A differential polynomial map $\varphi : W_1\to W_2$ between Kolchin closed subsets of $\U^{n_1}$ and $\U^{n_2}$, respectively, defined over $\K$, is given in coordinates by differential polynomials in
$\K\{W_1\}$. Moreover, to give $\varphi : W_1 \to W_2$
is equivalent to defining $\varphi^* : \K\{W_2\} \to \K\{W_1\}$.

\begin{definition}\cite[Chapter II, Section 1, page 905]{Cassidy}\label{def:LDAG} A {\it linear differential
algebraic group} 
is a Kolchin closed subgroup $G$ of $\GL_n(\U),$
that is, an intersection
of a Kolchin closed subset of $\U^{n^2}$ with $\GL_n(\U)$ that is closed under
the group operations.
\end{definition}

Again, in what follows, LDAG stands for linear differential algebraic group. Note that we identify $\GL_n(\U)$ with the Zariski closed
subset of $\U^{n^2+1}$ given by
$$\left\{(M,a)\:\big|\: (\det(M))\cdot a-1=0\right\}.$$
If $X$ is an invertible $n\times n$ matrix, we
can identify it with the pair $(X,1/\det(X))$. Hence, we may represent the coordinate ring of $\GL_n(\U)$ as
$
\K\{X,1/\det(X)\}$.
Denote $\GL_1$ simply by $\Gm$, called the multiplicative group. Its coordinate ring
is $\K\{y,1/y\}$. 
The LDAG with coordinate ring $\K\{y\}$ is denoted by $\Ga$, called the additive group. Finally, $\SL_2$ is the LDAG with the coordinate ring   $$\K\{c_{11},c_{12},c_{21},c_{22}\}/[c_{11}c_{22}-c_{12}c_{21} -1],$$
where the differential ideal of the quotient is radical because of \cite[Lemma~3.4]{Buium1993}.

\begin{definition}\cite{CassidyRep},\cite[Definition~6]{OvchRecoverGroup} Let $G$ be a LDAG. A differential polynomial
group homomorphism  $$r : G \to \GL(V)$$ is called a
{\it differential representation} of $G$, where $V$ is a
finite-dimensional vector space over $\K$. Such space is
simply called a {\it $G$-module}. This is equivalent to giving a comodule structure $$\rho : V \to V\otimes_\K \K\{G\}, $$ see \cite[Definition~7 and Theorem~1] {OvchRecoverGroup}. 

As usual, {\it morphisms} between $G$-modules are $\K$-linear maps that are $G$-equivariant. The category of differential representations of $G$ is denoted by $\Rep G$.
\end{definition}

\begin{remark} We will be going back and forth between the module and comodule terminology depending on the situation. The comodule language is needed primarily to avoid unnecessary extensions of scalars from $\K$ to $\U$ as our main classification result is over $\K$. 
\end{remark}

By \cite[Proposition~7]{Cassidy}, $r(G)\subset\GL(V)$ is a differential algebraic subgroup. 
Given a representation $r$ of a LDAG $G$, one can define its prolongation 
$F(r) : G \to \GL(FV)
$ with respect to $\partial$ as follows~\cite[Definition 4 and Theorem 1]{OvchRecoverGroup}: let 
\begin{equation}\label{eq:prolongation}
F(V) =\leftidx{_{\K}}{\left((\K\oplus \K\partial)_{\K}\otimes_{\K} V\right)}
\end{equation}
as vector spaces (see \cite[Section~4.3]{difftanncat} for a coordinate-free definition). Here, $\K\oplus \K\partial$ is considered as the right $\K$-module:
$
\partial\cdot a = \partial(a) + a\partial
$
for all $a \in \K$.
Then the action of $G$ is given by $F(r)$ as follows:
$$
F(r)(g) (1\otimes v) := 1\otimes r(g)(v),\quad F(r)(g)(\partial\otimes v) := \partial\otimes r (g)(v)
$$
for all $g \in G$ and $v \in V$. In the language of matrices, if $A_g \in \GL_n$ corresponds to the action of $g \in G$  on $V$, then the matrix
$$
\begin{pmatrix}
A_g&\partial A_g\\
0&A_g
\end{pmatrix}
$$
corresponds to the action of $g$ on $F(V)$. 

\section{Preparation}\label{sec:preparation}
Let $G$ be a group. In this section, we will recall some general terminology and basic facts that are useful to study non-semisimple categories of representations, that is, when not every $G$-module decomposes into a direct sum of irreducibles. This is precisely what we need to be able to handle to study differential representations of  LDAGs to obtain the main result of the paper in Section~\ref{sec:main}. 

\subsection{The set $\Rep_0G$ and its use}\label{sec:Rep0}
We start by introducing a special subset of representations $\Rep_0G$ and show how the rest of the representations can be reconstructed from it.
Since every $G$-module is a sum of indecomposable ones, it suffices to describe indecomposable modules.
As we will see below, it is possible to restrict ourselves to even a smaller subset of representations so that: 
\begin{itemize}
\item we are still able to recover all representations from it using only a few operations of linear algebra, namely pull-backs and push-outs, but not using $\otimes$, for instance, which is important for computation;
\item this set itself is much easier to describe.
\end{itemize}

\begin{definition} For an abstract group $G$, let $\Rep_0G$ be the set of all finite-dimensional $G$-modules $V$ having a unique minimal and a unique maximal submodules. The set $\Irr G$ of all simple $G$-modules is a subset of $\Rep_0G$ and every $V\in\Rep_0G$ is indecomposable (since otherwise $V$ has at least two minimal submodules).
\end{definition}

\begin{definition}
A $G$-module $V$ is said to be a {\it pull-back} of $V_1,V_2\in\Ob(\Rep G)$ if there is a $G$-module $W$ with surjections $\pi_k: V_k\to W$, $k=1,\:2$, such that $V$ is isomorphic to the pull-back of the maps $\pi_1$ and $\pi_2$. 

We say that $V$ is a {\it push-out} of $G$-modules $V_1$ and $V_2$ if there is a $G$-module $W$ with embeddings $\iota_k: W\to V_k$, $k=1,2$, such that $V$ is isomorphic to the push-out of the maps $\iota_1$ and $\iota_2$.
\end{definition}

\begin{proposition}\label{pull-push}
Every finite-dimensional $G$-module $V$ can be obtained from $\Rep_0G$ by iterating pull-backs and push-outs.
\end{proposition}
\begin{proof}
Suppose that $V\not\in\Rep_0G$ has two distinct minimal submodules $U_1$ and $U_2$. Set $$V_k:=V/U_k,\ \ k=1,\:2,\quad \text{and}\quad W:=V/(U_1+U_2).$$ Then $V$ is the pull-back of the corresponding (surjective) maps $\pi_k: V_k\to W$, $k=1,2$. Indeed, since $U_1\cap U_2=0$, $V$ embeds into the pull-back $$V_{12}:=\{(v_1,v_2)\in V_1\times V_2\colon \pi_1(v_1)=\pi_2(v_2)\}.$$  On the other hand, if $\overline{v_k}\in V_k$, $k=1,\:2$, and $\pi_1(\overline{v_1})=\pi_2(\overline{v_2})$, then there are $v_1,\:v_2\in V$ such that $$v_1+U_1+U_2=v_2+U_1+U_2.$$ Hence, $$v_1+u_1=v_2+u_2=:v\in V\quad \text{for some}\ \ u_k\in U_k,\ \ k=1,\:2.$$ This shows that $\overline{v_k}$ is the image of $v$ under the quotient map $V\to V_k$, $k=1,\:2$. Hence, $V\simeq V_{12}$. 

Now suppose that $V$ has two distinct maximal submodules $V_1$ and $V_2$. Let $$U:=V_1\cap V_2\quad \text{and}\quad \iota_k: U\to V_k,\ \ k=1,\:2,$$ be the corresponding embeddings. Then $V$ is isomorphic to the push-out of the maps $\iota_1$ and $\iota_2$. Indeed, let $W$ be a $G$-module with morphisms $\alpha_k: V_k\to W$, $k=1,2$, such that $\alpha_1\iota_1=\alpha_2\iota_2$. Since $V=V_1+V_2$, this implies that the morphism $\alpha: V\to W$ given by  $$\alpha(v_1+v_2)=\alpha_1(v_1)+\alpha_2(v_2)\quad\text{for all $v_k\in V_k,\ \ k=1,\:2$}$$ 
is well-defined. Hence, $V$ is the push-out.
Finally, the statement of the proposition follows by induction on $\dim V$.
\end{proof}

Pull-backs and push-outs have a simple description in terms of matrices. This is why Proposition \ref{pull-push} is particularly useful in computation. Namely, if $\pi_k: V_k\to W$, $k=1,2$, are the surjections, then we can choose bases of $V_1$ and $V_2$ such that every $g\in G$ is represented in $\GL(V_1)$ and $\GL(V_2)$ by matrices of the form
$$
\begin{pmatrix}
A(g) & B(g)\\
0 & C(g)
\end{pmatrix}
\quad\text{and}\quad 
\begin{pmatrix}
A_1(g) & B_1(g)\\
0 & C(g)
\end{pmatrix},
$$
where $C(g)$ corresponds to the representation $G\to\GL(W)$.
Then the pull-back $V$ of $\pi_1$ and $\pi_2$ has the following matrix structure:
$$
G \ni g \mapsto
\begin{pmatrix}
A(g) & 0 & B(g)\\
0 & A_1(g) & B_1(g)\\
0 & 0 & C(g)
\end{pmatrix}.
$$
In terms of bases, if $V_1= \Span\{E_1,E_2\}$ and $V_2 = \Span\{F_1,F_2\}$, where $E_i$'s and $F_i$'s are the sequences of basis elements corresponding to the block structure, then $V$ can be viewed as $$\Span\{E_1,F_1,E_2+F_2\} \subset V_1\oplus V_2,$$ where $E_2+F_2$ means the sum of the corresponding basis elements.

If $\iota_k: U\subset V_k$, $k=1,2$, are embeddings, we can choose bases of $V_1$ and $V_2$ such that every $g\in G$ is represented in $\GL(V_1)$ and $\GL(V_2)$ by matrices of the form
$$
\begin{pmatrix}
A(g) & B(g)\\
0 & C(g)
\end{pmatrix}
\quad\text{and} 
\begin{pmatrix}
A(g) & B_1(g)\\
0 & C_1(g)
\end{pmatrix}.
$$
where $A(g)$ corresponds to the representation $G\to\GL(U)$. Then the push-out $V$ of $\iota_1$ and $\iota_2$ has the following matrix structure:
$$
G\ni g \mapsto
\begin{pmatrix}
A(g) & B(g) & B_1(g)\\
0 & C(g) & 0\\
0 & 0 & C_1(g)
\end{pmatrix}.
$$

\subsection{Simple socle}\label{sec:ss}
The observations from this section will be further used in Section~\ref{sec:main} to prove our main result.
Recall that the \emph{socle} $V_{soc}$ of a $G$-module $V$ is the smallest submodule of $V$ containing all simple submodules of $V$. In particular, if $V$ is finite-dimensional, $V_{soc}$ is the sum of all simple submodules of $V$. If $V_{soc}$ is simple, it is a unique minimal submodule of $V$. Conversely, if $V$ contains a unique minimal submodule, $V_{soc}$ is simple (and coincides with the submodule). Any $V\in\Rep_0G$ has a simple socle.

\begin{remark}
There are two alternative definitions of the set $\Rep_0 G$:
\begin{enumerate}
\item $\Rep_0G$ is the smallest set $S$ of $G$-modules with the property that every finite-dimensional $G$-module is obtained from $S$ by a sequence of pull-backs and push-outs.
\item $\Rep_0G$ is the set of $G$-modules $V$ such that $V$ and $V^\vee$ have simple socles.
\end{enumerate}
\end{remark}

\begin{proposition}\label{inj}
Let $V$ be a $G$-module with simple socle and $\alpha\colon V\to W$ a morphism of $G$-modules such that $\alpha(V_{soc})\neq 0$. Then $\alpha$ is injective. Moreover, if $W=\prod_{i\in I}W_i$, then there exists $i\in I$ such that $\pi_i\alpha$ is an isomorphism of $V$ and a submodule of $W_i$, where $\pi_i: W\to W_i$ is the projection. 
\end{proposition}
\begin{proof}
If the submodule $\Ker\alpha\subset V$ is non-zero, it must contain $V_{soc}$, the smallest submodule of $V$. Since $V_{soc}\not\subset\Ker\alpha$, we have $\Ker\alpha=0$. To prove the second part of the statement, note that there is an index $i\in I$ such that $\pi_i(\alpha(V_{soc}))\neq 0$. Then we apply the first part of the statement to the map $\pi_i\alpha:V\to W_i$.   
\end{proof}

Let $G$ be a LDAG.
Its coordinate ring $A := \K\{G\}$ has a structure of a differential Hopf algebra, that is, a Hopf algebra in which the comultiplication, antipode, and counit are homomorphisms of differential algebras \cite[Section~3.2]{OvchRecoverGroup} and \cite[Section~2]{CassidyRep}.  Let $$\Delta : A \to A\otimes_\K A$$ be the comultiplication inducing the 
right-regular $G$-module structure on $A$ as follows (see also \cite[Section~4.1]{OvchRecoverGroup}). For $g, x \in G(\U)$ and $f \in A$,
$$
\left(r_g(f)\right)(x)=f(x\cdot g) = \Delta(f)(x,g)= \sum_{i=1}^nf_i(x) g_i(g),
$$
where $\Delta(f) = \sum_{i= 1}^n f_i\otimes g_i$.

\begin{proposition}\label{reg}\label{prop:reg}
Every finite-dimensional $G$-module $V$ with simple socle embeds into the regular functions $A$. 
\end{proposition}
\begin{proof}
By~\cite[Lemma~3]{OvchRecoverGroup}, $V$ embeds into $A^{\dim V}$. Now the statement follows from Proposition~\ref{inj}.
\end{proof}

\begin{example}\label{ex:reg} Let $V = \Span_\K\left\{x^2,xy,y^2,x'y-xy'\right\} \subset \K\{x,y\}$ and the action of $\SL_2$ is given by 
$$\SL_2(\U) \ni \begin{pmatrix}
a& b\\
c&d
\end{pmatrix} \mapsto
\begin{cases}
x^{(p)}\mapsto (ax+cy)^{(p)},\ p\Ge 0;\\
y^{(q)} \mapsto (bx + dy)^{(q)},\ q\Ge 0.
\end{cases}
$$
with the coordinate ring $$A:=\K\{x_{ij}\}/[x_{11}x_{22}-x_{12}x_{21}-1].$$ 
Hence, for the induced $A$-comodule structure $\rho_V : V\to V\otimes_\K A \cong A^4$,
\begin{align*}
x^2 &\mapsto x^2\otimes x_{11}^2 + xy\otimes 2x_{11}x_{21} + y^2\otimes x_{22}^2 ,\\
 xy&\mapsto  x^2\otimes x_{11}x_{12} + xy\otimes (x_{11}x_{22}+x_{12}x_{21}) + y^2\otimes x_{11}x_{21},\\
y^2 &\mapsto x^2\otimes x_{12}^2 + xy\otimes 2x_{12}x_{22} + y^2\otimes x_{22}^2,\\
x'y-xy'&\mapsto
x^2\otimes  \left(x_{11}'x_{12}-x_{11}x_{12}'\right) + xy\otimes 2\left(x_{11}'x_{22}-x_{12}'x_{21}\right) + y^2\otimes \left(x_{21}'x_{22}-x_{21}x_{22}'\right) + \left(x'y-xy'\right)\otimes 1.
\end{align*}
Since the projection $A^4 \to A$ onto the first coordinate (i.~e. the coefficient of $x^2$) is non-zero on $$V_{soc} = \Span_\K\left\{x^2,xy,y^2\right\},$$ this projection is injective on the whole $V$, and the image is $$\Span_\K\left\{x_{11}^2,x_{11}x_{12},x_{12}^2,x_{11}'x_{12}-x_{11}x_{12}'\right\} \subset A$$ (see also \cite[Remark~4.9]{diffreductive}).
\end{example}

By a subquotient of $V$, we mean a $G$-module $V_1/V_2$ where $V_2\subset V_1$ are submodules of $V$.
The following recalls  a way of describing categories of representations in which not every representation is a direct sum of irreducibles \cite[Section~I.4.1]{Bernstein}.

\begin{definition}\label{def:JHsplitting}
For any $V\in\Ob(\Rep G)$, denote  the set of all simple subquotients of $V$ by $\JH(V)$. For a subset $S\subset\Irr G$, we say that $V\in\Ob(\Rep G)$ is $S$-{\it isotypic}, if $\JH(V)\subset S$. 

We say that $S$ is {\it splitting} if any $V$ is a direct sum $U\oplus W$, where $\JH(U)\subset S$ and $\JH(W)\cap S=\varnothing$. 
\end{definition}
By definition, the set $\Irr G$ is splitting for $\Rep G$. For each $G$, the goal is to find as small splitting sets as possible. We will see in Proposition~\ref{splitting} that tori have splitting sets consisting just of one representation. 

The following statement will be further used in Section~\ref{sec:main}. 
\begin{proposition}\label{subquot}
Let $G\subset\GL_n$ be a LDAG defined over $\Q$ by polynomials (of order zero) and $V$ be a  $G$-module. Then every simple subquotient $U$ of $V$ is a usual (non-differential) representation of $G$ considered as a linear algebraic group. 
Moreover, simple $G$-modules are isomorphic if and only if they are isomorphic as $G(\Const)$-modules. Finally, if $G$ is reductive, then, as a $G(\Const)$-module, $V$ is a direct sum of its simple submodules.
\end{proposition}
\begin{proof}
By \cite[Theorem~3.3]{diffreductive}, $U$ is algebraic. The second statement of the proposition follows from the fact that $G(\Const)$ is Zariski-dense in $G$, because it is given by polynomial equations over $\Q$, $\Q \subset \Const$, and $\Const$ is algebraically closed  \cite[Corollary~AG.13.3]{Borel}. Since the group $G(\Const)$ is reductive, $V$ is completely reducible as a $G(\Const)$-module (see, for example, \cite[Chapter 2]{SpringerInv}). 
\end{proof}

\section{Differential representations of $\Gm^n$ and $\SL_2$}\label{sec:main}
We will start by describing differential representations of the additive and multiplicative groups in Sections~\ref{sec:Ga} and~\ref{sec:Gm}, which we give here for comparison, and then show our main result on differential representations of $\SL_2$ in Section~\ref{sec:SL2}, where the situation is very different from the vector groups and tori.

\subsection{Differential representations of $\Ga^n$}\label{sec:Ga}
As usual, for a nilpotent matrix $N$ with entries in $\K$, we define $\exp(N) = \sum_{i=0}^\infty N^i/i!$. The following result not only characterizes differential representations of the additive group but is also used to describe all differential representations of tori in Theorem~\ref{thm:tori}.
\begin{proposition}\label{unip_diff}
A finite array $N=\left\{N_{i,j}\:\big|\: 1\Le i \Le n,\ j=0,1,2,\dots\right\}$ of mutually commuting nilpotent $r\times r$ matrices  with entries in $\K$ defines a LDAG homomorphism $$\alpha_N: \Ga^n\to\GL_r,\quad (x_1,\dots,x_n)\mapsto\exp\left(\sum_{j=0}^\infty\sum_{i=1}^n N_{i,j}\partial^jx_i\right).$$ Any differential representation $\alpha: \Ga^n\to\GL_r$ (over $\K$) is equivalent to $\alpha_N$ for some $N$. The representations $\alpha_N$ and $\alpha_M$ are equivalent if and only if there exists $Q\in\GL_r(\K)$ such that $M_{i,j}=QN_{i,j}Q^{-1}$ for all $i$ and $j$.
\end{proposition}

\begin{proof}
It is straightforward that $\alpha_N$ is a differential representation. Now let $\alpha: \Ga^n\to\GL_r$ be a differential representation. If $k$ is the largest order of a matrix entry of $\alpha$, then there exists an algebraic representation $\beta: \Ga^{kn}\to\GL_r$ such that $$\alpha(x_1,\dots,x_n)=\beta{\left(x_1,\partial x_1,\dots, \partial^k x_1,x_2,\dots,\partial^kx_n\right)}.$$ 
Indeed, let $$\rho_\alpha : \K^r\to \K^r\otimes_\K\K\{x_1,\ldots,x_n\},\quad e_j \mapsto\sum_{i=1}^re_i\otimes a_{ij},\ 1\Le j\Le r,$$ where $\{e_1,\ldots,e_r\}$ is the standard basis of $\K^r$, be the comodule structure corresponding to $\alpha$. Then $$a_{ij} \in \K{\left[x_1,\partial x_1, \dots, \partial^k x_1,x_2,\dots,\partial^kx_n\right]},\quad 1\Le i,j\Le r.$$ Now, $\beta$ is defined to be the linear algebraic group homomorphism corresponding to the (same) comodule structure
$$
\rho : \K^r\to \K^r\otimes_\K\K{\left[x_1,\partial x_1,\dots, \partial^k x_1,x_2,\dots,\partial^kx_n\right]},\quad e_j \mapsto\sum_{i=1}^re_i\otimes a_{ij},\ 1\Le j\Le r.
$$ 

There are mutually commuting nilpotent matrices $N_i^j$, $1\le i\le n$, $0\le j\le k$, such that $$\beta{\left(\left\{\partial^jx_i\right\}\right)}=\exp{\left(\sum_{j=0}^k\sum_{i=1}^n N_{i,j}\partial^jx_i\right)}$$ (see, for instance, \cite[Theorem~12.3.6]{Crumley}). Thus, $\alpha=\alpha_N$, where $N=\{N_{i,j}\}$. The last statement follows from the definition of $\alpha_N$, that is, $\exp$ commutes with conjugation, and the linear independence of $\left\{\partial^jx_i\right\}$.
\end{proof}

\subsection{Differential representations of $\Gm^n$}\label{sec:Gm}
In Section~\ref{sec:Gmgeneral}, we will characterize all differential representations of tori. Then, Section~\ref{sec:Gmdiffpoly} contains the results on the action of $\Gm$ on differential polynomials that we further use in  Section~\ref{sec:SL2} to prove our main result.
\subsubsection{General characterization.}\label{sec:Gmgeneral}
In this section, we study the category $\Rep\Gm^n$.
Recall that $\Irr\Gm^n$ consists of the characters $$\chi^d\colon\Gm^n\to\Gm,\quad (x_1,\dots,x_n)\mapsto x_1^{d_1}\cdot\ldots\cdot x_n^{d_n},\quad d=(d_1,\dots,d_n)\in\Z^n,$$
because any irreducible representation of a LDAG can be given by polynomials (without any derivatives involved) by \cite[Theorem~3.3]{diffreductive}, and, therefore, \cite[Proposition~8.5]{Borel} gives the result.
We will regard $\Gm^n$ as a subgroup of $\GL_{2n}$, so that its coordinate ring is (due to~\cite[Lemma~3.4]{Buium1993}, we do not have to take the radical) $$A:=\K\{\Gm^n\}=\K\{x_1,y_1,\ldots,x_n,y_n\}/[x_1y_1-1,\ldots,x_ny_n-1].$$  

\begin{proposition}
\label{splitting}
Every element of $\Irr\Gm^n$ is splitting (see Definition~\ref{def:JHsplitting}).
\end{proposition}
\begin{proof}
Set $G:=\Gm^n$. Let $V$ be a $G$-module. It follows from Proposition~\ref{subquot} that 
$$V = \bigoplus_{d\in\Z^n}V_d,\quad V(d):=\left\{v\in V\:\big|\: g(v)=\chi^d(g)v\ \; \text{for all }g\in G(\Const)\right\}.$$
Since $G$ is commutative, $V(d)$ is $G$-invariant.
We conclude that $V$ is the direct sum of its $\chi^d$-isotypic components $V(d)$ for all $d\in\Z^n$.
\end{proof}

Consider the logarithmic derivative homomorphism (see~\cite[page~924]{Cassidy} and~\cite[page~648]{SitSL2}): $$\lambda : \Gm^n\to\Ga^n,\quad (x_1,\dots,x_n)\mapsto \left(x_1'y_1,\dots,x_n'y_n\right).$$ For every representation $\alpha : \Ga^n\to\GL(V)$, we have the representation $\alpha\circ\lambda : \Gm^n\to\GL(V)$.

\begin{theorem}\label{thm:tori}
\label{rep_Gm}
Any differential representation $\beta: \Gm^n\to\GL_r$ is isomorphic to the direct sum of its $\chi^d$-isotypic components $$\beta_d: \Gm^n\to\GL_{r_d},\quad  x\mapsto\chi^d(x)\cdot\alpha_d(\lambda(x)),$$ where $\alpha_d : \Ga^n\to\GL_{r_d}$ is a LDAG homomorphism and $d\in\Z^n$. Representations $$\beta,\:\beta'\colon\Gm^n\to\GL_{r_d}$$ are equivalent if and only if the corresponding $\alpha_d,\alpha'_d$ are equivalent for all $d\in \Z^n$ (see Proposition~\ref{unip_diff}).
\end{theorem}
\begin{proof}
By Proposition \ref{splitting}, $\beta$ is the direct sum of its isotypic components. Hence, we may assume that $\beta$ is $\chi^d$-isotypic for some $d\in\Z^n$. Moreover, tensoring $\beta$ with $\chi^{-d}$, we may assume that $d=0$. Then the image of $\beta$ consists of unipotent matrices. Since $\beta(\Gm^n(\Const))$ is diagonalizable, we have $\Gm^n(\Const)=\Ker\lambda\subset\Ker\beta$. Since $\lambda$ is onto, the homomorphism theorem for LDAGs  \cite[page 917]{Cassidy} implies that $$\beta=\alpha\circ\lambda$$ for some LDAG homomorphism $\alpha : \Ga^n\to\GL_r$ (defined over $\K$). 
\end{proof}

\subsubsection{Action of $\Gm$ on differential polynomials.}\label{sec:Gmdiffpoly} What follows in this section will be further used in Section~\ref{sec:SL2}, in particular, in Lemma~\ref{lem:firstorder}, to characterize differential representations of $\SL_2$ that are extensions of two irreducible representations. We will additionally suppose that $\K$ has a non-constant element.

Let the group $\Gm$, with its differential Hopf algebra $\K\{z,1/z\}$, act on the differential polynomial algebra $P:=\K\{x,y\}$ via the comodule structure
$$\rho: P \to P\otimes_\K\K\{z,1/z\},\quad x\mapsto x\otimes z,\ y\mapsto y\otimes 1/z.$$ Let $M$ be the set of all terms (a term is a product of a coefficient from $\K$ and a monomial)  in $P$. For $f\in P$, denote  the set of all terms that are present in $f$ by $M(f)$. For $$0\ne a \in\K\quad \text{and}\quad \rho(f)=\sum_i f_i\otimes b_i,\ f_i\in P,\ b_i\in A,$$ we let
$$
\rho(f)(a) := \sum_i b_i(a)f_i \in P.
$$
For a term 
\begin{equation}\label{eq:monomial}
h=\alpha\cdot{\left(x^{(p_1)}\right)}^{m_1}\cdot\ldots\cdot{\left(x^{(p_k)}\right)}^{m_k} \cdot {\left(y^{(q_1)}\right)}^{n_1}\cdot\ldots\cdot{\left(y^{(q_t)}\right)}^{n_t},
\end{equation}
where $p_i,m_i,q_j,n_j$ are non-negative integers, $p_1<\ldots < p_k$, $q_1 <\ldots < q_t$, and $0\neq\alpha\in\K$, its \emph{weight} is, by definition, \begin{equation}\label{eq:wtdef}
\sum p_im_i+\sum q_jn_j.
\end{equation} We also set 
\begin{equation}\label{eq:dhdef}
d(h):=\sum_im_i-\sum_jn_j.
\end{equation}
The weight $\wt f$ of an element $f\in P$ is defined as the maximum over the weights of all $h\in M(f)$. Note that, for any $f\in P$, $\wt f = 0$ if and only if $f\in\K[x,y]$.

Let $S$ be the set of all finite sequences $u=(u_0,u_1,\dots)$ of non-negative integers. We define a total ordering on $S$ by
$$
u<v\quad\Longleftrightarrow\quad\text{for the maximal $i$ such that $u_i\neq v_i$, we have $u_i<v_i$}. 
$$  
The total ordering on $S\times S$ is defined by
$$
(u,\tilde u)<(v,\tilde v)\quad\Longleftrightarrow\quad\text{$\tilde u<\tilde v$ or ($\tilde u=\tilde v$ and $u<v$)}. 
$$
To every $h\in M$, we assign a pair $s(h)=(u,v)\in S\times S$, where $u_i$ (respectively, $v_i$) is the multiplicity in $h$ of the factor $x^{(i)}$ (respectively, $y^{(i)}$). 

Thus, we have established a bijection between $\overline M = M/\sim$ and $S\times S$, where the equivalence $h\sim f$ means $f=\alpha h$ for some $0\neq\alpha\in\K$. We transfer the total ordering from $S\times S$ to $\overline M$. For any $h,f\in M$, we write $h<f$, and say that $h$ is smaller than $f$, if $s(h)<s(f)$; see also \cite{ZobninEss} for differential monomial orderings. 

\begin{lemma}\label{lem:max}
For every $h\in M$ with $\wt h > 0$ and $a\in\K$ with $a'\neq 0$, we have $$\wt{\left(\rho(h)(a)-a^{d(h)}h\right)}=\wt(h) - 1.$$
Moreover, there exists  $$\tilde h\in M{\left(\rho(h)(a)-a^{d(h)}h\right)}$$ such that $\tilde h<h$ and, for all $f\in M$ with $f<h$ and $d(f)=d(h)$, we have either $f<\tilde h$ or $f\sim \tilde h$.  
\end{lemma}
\begin{proof}
Suppose $h$ is given by~\eqref{eq:monomial} and there is an index $i$ with $p_i>0$. Then we may assume that $i$ is the smallest index with this property. Let $$h_i = \frac{h}{{\left(x^{(p_i)}\right)}^{m_i}}.$$ We set 
$$
\tilde h = \alpha\cdot m_i\cdot p_i\cdot a^{d-1}\cdot a'\cdot {\left(x^{(p_i)}\right)}^{m_i-1}\cdot x^{(p_i-1)}\cdot h_i,\quad d=d(h).
$$
We have
$$
\rho(h)(a) = \rho(h_i)(a)\cdot {\left((ax)^{(p_i)}\right)}^{m_i}= a^d h +\tilde{h}+\ldots,
$$
where $\ldots$ is a sum of terms that are smaller than $\tilde h$ and have weights $<\wt h$. The rest of the  properties of $\tilde h$ follow  from its definition. In the case when all $p_i$'s are zeros, since $\wt h>0$, there is an index $j$ with the property $q_j>0$. Then we choose the smallest such $j$ and define $\tilde h$ by replacing $x$ by $y$, $i$ by $j$, $p$ by $q$, and $m$ by $n$.  
\end{proof}

\begin{lemma}\label{lem:free} Let $\K$ have a non-constant element $a$.
If $V\subset P$ is a $\Gm$-submodule containing an element with positive weight $w$, then $V$ also contains an element with weight $w-1$.
\end{lemma}
\begin{proof}
Since $\Gm(\Const)$ is an algebraic torus, $$V=\bigoplus_{d=-\infty}^\infty V(d),\quad V(d):=\left\{v\in V\:\big|\: \rho(v)(b)=b^dv\ \; \text{for all }0\neq b\in\Const\right\}.$$  By the assumption, there exists $f\in P$ with $\wt f = w$. Hence, there is  $h\in M(f)$ such that $\wt h = w$. Since $h$ is a term, $$h\in V(d(h)),$$ see~\eqref{eq:dhdef}. Then the sum of all terms in $M(f)$ lying in $V(d(h))\subset V$ has weight $w$. 

Now suppose that $f\in V(d)$ and $\wt f = w$. We claim that $$\wt\left(\rho(f)(a)-a^df\right)=w-1.$$ Indeed, let $f_w$ be the sum of the elements of $M(f)$ of weight $w$. We have
$$
f=f_w+f_{<w},
$$   
where $f_{<w}$ is the sum of the elements of $M(f)$ of weight $\Le w-1$.
Let $h$ be the maximal element of $M(f_w)$ and $g=f_w-h$. 
We have
$$
\rho(f)(a)-a^df=\left(\rho(h)(a)-a^dh\right)+\left(\rho(g)(a)-a^dg\right)+\left(\rho(f_{<w})(a)-a^df_{<w}\right).
$$
Let $$\tilde h\in M{\left(\rho(h)(a)-a^dh\right)}$$ be the element defined by Lemma~\ref{lem:max}. Then $\wt\tilde h=w-1$. We will show that $\tilde h$ is not equivalent to an element of $$M{\left({\left(\rho(g)(a)-a^dg\right)}+{\left(\rho{\left(f_{<w}\right)}(a)-a^df_{<w}\right)}\right)},$$ which will finish the proof. Let $$p \in M{\left(\rho{\left(f_{<w}\right)}(a)-a^df_{<w}\right)}.$$ By Lemma~\ref{lem:max}, $\wt p\le w-2$ and, therefore, $p\not\sim\tilde h$. Now let $$p\in M{\left(\rho(g)(a)-a^dg\right)}.$$ There exists $g_0\in M(g)$ such that $$p\in M{\left(\rho(g_0)(a)-a^dg_0\right)}.$$ Then $p<g_0<h$. By Lemma~\ref{lem:max}, either $g_0<\tilde h$ or $g_0\sim\tilde h$. In any case, then $p<\tilde h$.   
\end{proof}

\subsection{Main result: differential extensions of irreducible representations of $\SL_2$}\label{sec:SL2}
Theorem~\ref{thm:tori} shows that an extension of two non-isomorphic
irreducible representations of a torus splits. As we have seen in
Example~\ref{ex:reg}, this is not true for differential representations of $\SL_2$. In particular, one could form differential extensions of representations of different dimensions, and therefore, non-isomorphic. In this section, we will show how to handle this situation and provide a characterization of all differential $\SL_2$-modules that are extensions of any two irreducible $\SL_2$-modules.

 As announced in the introduction, in this section, we also {\it additionally suppose} that there exists $a \in \K$ with $a' \ne 0$. We need this extra assumption only in the proof of Lemma~\ref{lem:firstorder} below, which refers to Lemma~\ref{lem:free}, where this condition is explicitly used.
Our description will consist of several steps. We will call $\SL_2$ by $G$ from time to time.
Let (again, we do not have to take the radical due to \cite[Lemma~3.4]{Buium1993} and \cite{Mustata}) 
\begin{equation}\label{eq:ABP}
C=\K\{c_{ij}\}_{1\Le i,j\Le 2},\ \  \ \det=c_{11}c_{22}-c_{12}c_{21},\ \ \ A=\K\{G\}=C/[\det-1],\ \  \ B= C/[\det],\ \ \ P=\K\{x,y\}
\end{equation}
with the action of $\SL_2$ derived from the one given in Example~\ref{ex:reg}.

The proof of the following lemma, which we will use in the proofs of Lemma~\ref{lem:degree} and Theorem~\ref{thm:embedintoB}, is due to M.~Kondratieva.
\begin{lemma}\label{lem:detprime} The differential ideal $\left[\det'\right] \subset C$ is prime (see~\eqref{eq:ABP}).
\end{lemma}
\begin{proof}
We will first show that \begin{equation}\label{eq:detprimeequals}
\left[\det{}'\right] = \left[\det{}'\right]:c_{11}^\infty := \left\{f \in C\:|\:\text{there exists } n\Ge 0 \text{ such that } c_{11}^n\cdot f \in \left[\det{}'\right]\right\}.
\end{equation}
By definition, $\left[\det{}'\right] \subset \left[\det{}'\right]:c_{11}^\infty$. To show the reverse inclusion, we will prove that, for all $q \Ge 1$, we have
\begin{equation}\label{eq:detq}
\left[\det{}'\right] \supset \left(\det{}',\det{}'',\ldots,\det{}^{(q)}\right):c_{11}^\infty.
\end{equation}
To show~\eqref{eq:detq}, it is enough to prove that, for all $q \Ge 1$,
$$
I  \cap C= \left(\det{}',\det{}'',\ldots,\det{}^{(q)}\right),\quad I := \left(\det{}',\det{}'',\ldots,\det{}^{(q)},1-t\cdot c_{11}\right)\cdot C[t].
$$
For this, it is enough to show that the set of elements of a Gr\"obner basis of $I$ not depending on $t$ with respect to a monomial ordering such that $t >_{\mathrm{lex}}$ than any other variable (any ordering that eliminates $t$) is equal
to  $$G := \det{}',\det{}'',\ldots,\det{}^{(q)},$$
\cite[Exercise~4.4.9]{CoxLittleOShea}. To do this, we choose the grevlex monomial ordering \cite[Definition~2.2.6]{CoxLittleOShea} on $C[t]$ with $$t>c_{22}^{(q)}>c_{21}^{(q)}>c_{12}^{(q)}>c_{11}^{(q)}>\ldots>c_{22}>c_{21}>c_{12}>c_{11}.$$
Since, for all $i$, $1\Le i\Le q$, the leading monomial of $\det{}^{(i)}$ is 
$$
\begin{cases}
c_{11}^{(k+1)}\cdot c_{22}^{(k)}&  i=2k+1,\ k\Ge 0;\\  
c_{12}^{(k)}\cdot c_{21}^{(k)}& i =2k,\ k\Ge 1,
\end{cases}
$$
we conclude that the leading monomials in $\tilde{G} := G \cup \left\{1-t\cdot c_{11}\right\}$ are relatively prime. Therefore, $\tilde{G}$ is a Gr\"obner basis of $I$ by \cite[Theorem~2.9.3 and Proposition~2.9.4]{CoxLittleOShea} and $\tilde{G}\cap C = G$. Thus, we have~\eqref{eq:detprimeequals}.

Finally, since $\det{}'$ is an irreducible differential polynomial, \cite[Lemma~IV.9.2]{Kol} implies that $[\det{}']:c_{11}^\infty$ is a prime differential ideal (see also \cite[Theorem~4.7]{HubertEssential}).
\end{proof}

\begin{definition}
For  $f\in A$,  denote the smallest degree (the total degree when considered as a polynomial) of a representative in  $C$ by $\deg f$, which we also call the \emph{degree} of $f$. Similarly, we define the degree of $f\in B$. 
\end{definition}
\begin{remark}Note that $$A_{\Le d} := \Span_\K\{f\in A\:|\: \deg f\Le d\}$$ 
 is a $G$-module.
 \end{remark}

\begin{definition}
For  $w\in W\in\Ob(\Rep G)$, the degree $\deg w$ is the smallest $d\Ge 0$ such that,  for the comodule map $\rho_W : W\to W\otimes_\K A$, we have $$\rho_W(w) \in W\otimes_\K A_{\Le d}.$$ 
\end{definition}
The following lemma shows that, in the case $W\subset A$, our definitions of degree agree. We will use the notations $\pi_A$ and $\pi_B$ for the quotient maps $C\to A$ and $C\to B$, respectively. Let $C_d\subset C$ be the submodule of homogeneous differential polynomials of degree $d$ (considered as the usual polynomials) and $$C_{\Le d}=\bigoplus_{i=0}^dC_i.$$ We have $\pi_A(C_{\Le d})=A_{\Le d}$.
 
\begin{lemma}\label{lem:degree}
For the comultiplication $\Delta: A\to A\otimes_\K A$, the following hold:
\begin{equation}\label{eq:delta}
\Delta(A_{\Le d})\subset A_{\Le d}\otimes_\K A_{\Le d}
\end{equation}
and
\begin{equation}\label{eq:preimage}
\Delta^{-1}(A_{\Le d}\otimes_\K A_{\Le d-1}+A_{\Le d-1}\otimes_\K A_{\Le d})=A_{\Le d-1}.
\end{equation}
\end{lemma}
\begin{proof}
Here, we use the differential analogues \cite[Section~2]{CassidyRep} and \cite[Section~3]{OvchRecoverGroup} of the standard facts \cite[Sections~1.5,~3.2]{Waterhouse} on the relation between multiplicative structures on affine sets and bialgebra structures on their algebras of regular functions.
The group $G$ is a submonoid of the differential monoid $M$ of all $2\times 2$ matrices, defined similarly to Definition~\ref{def:LDAG}. This means that we have the following commutative diagram:
\begin{equation}\label{eq:diag1}
\begin{CD}
C @>\Delta_C>>C\otimes_\K C\\
@V{\pi_A}VV @VV{\pi_A\otimes\pi_A}V\\
A@>\Delta>>A\otimes_\K A 
\end{CD}
\end{equation}
where $\Delta_C$ is the comultiplication on the differential bialgebra $C$. For the generators $c_{ij}$, $1\Le i,j\Le 2$, of $C$, we have $$\Delta_C(c_{ij})=\sum_{k=1}^2c_{ik}\otimes c_{kj}.$$ This implies  
\begin{equation}\label{eq:note}
\Delta_C(C_d)\subset C_d\otimes_\K C_d,
\end{equation}
and, in view of \eqref{eq:diag1}, we obtain~\eqref{eq:delta}.  
Set $$I:=\Ker\pi_A=[\det-1]\quad \text{and}\quad J:=\Ker(\pi_A\otimes\pi_A)=I\otimes_\K C + C\otimes_\K I.$$ To prove \eqref{eq:preimage}, it suffices to show that if, for some $f\in C_{\Le d}$,  
\begin{equation}\label{eq:assume}
(\pi_A\otimes\pi_A)\Delta_C(f)\in A_{\Le d}\otimes_\K A_{\Le d-1}+A_{\Le d-1}\otimes_\K A_{\Le d},
\end{equation}
then $$f\in C_{\Le d-1}+I.$$ Moreover, since $C_{\Le d}=C_{\Le d-1}\oplus C_d$ and~\eqref{eq:note}, we only need to consider the case $f\in C_d$. Note that~\eqref{eq:note} and~\eqref{eq:assume} imply
\begin{equation}\label{eq:dcf}
\Delta_C(f)\in C_{\Le d}\otimes_\K C_{\Le d-1}+C_{\Le d-1}\otimes_\K C_{\Le d}+\tilde J,
\end{equation}
where $\tilde J := J\cap C_{\Le d}\otimes_\K C_{\Le d}$.
For the direct sum decomposition 
$$
C\otimes_\K C=\bigoplus_{i,j}C_{ij},\quad\text{where}\quad C_{ij}:=C_i\otimes_\K C_j,
$$
denote the projection onto $C_{ij}$ by $\pi_{ij}$. 
By \eqref{eq:note}, $$\Delta_C(f)=\pi_{dd}(\Delta_C(f)).$$ Then, by \eqref{eq:dcf}, we have $\Delta_C(f)\in\pi_{dd}{\left(\tilde J\right)}$, and, therefore,
\begin{equation}\label{eq:deltacf}
\Delta_C(f)=(\det-1)\cdot f_0\otimes a + b\otimes (\det-1)\cdot h_0 + g
\end{equation}
for some $$f_0,\;h_0,\;a,\;b\in C\quad \text{and}\quad g\in \left[\det{}'\right]\otimes_\K C + C\otimes_\K {\left[\det{}'\right]}.$$ Note that if $$\det\cdot f_0\in{\left[\det{}'\right]},$$ then, by Lemma~\ref{lem:detprime}, $$f_0\in{\left[\det{}'\right]}\subset[\det].$$ Hence, collecting terms of highest degree in~\eqref{eq:deltacf}, we obtain
\begin{equation}\label{eq:det}
\Delta_C(f)\in [\det]\otimes_\K C + C\otimes_\K [\det].
\end{equation}
We will show that then $f\in [\det]$, which means $$f=f_0\cdot\det+f_1\cdot\det{}'+\ldots+f_k\cdot\det{}^{(k)}$$ for some integer $k$ and $f_i\in C_{d-2}$, and, therefore, $$f\in f_0+I\subset C_{\Le d-2}+I\subset C_{\Le d-1}+I.$$ To this end, consider the (differential) subvariety $M_0\subset M$ of singular matrices. Since $M_0$ is closed under multiplication, the algebra $$\K\{M_0\}=C/[\det],$$ which is, again, reduced by \cite{Mustata}, inherits the comultiplication $\Delta_0$ from $C$. In other words, we have the commutative diagram
$$
\begin{CD}
C @>\Delta_C>>C\otimes_\K C\\
@V{\pi_0}VV @VV{\pi_0\otimes\pi_0}V\\
\K\{M_0\}@>\Delta_0>>\K\{M_0\}\otimes_\K \K\{M_0\} 
\end{CD}
$$
where $\pi_0$ is the quotient map. Then, in view of~\eqref{eq:det}, to prove $f\in [\det]$, it suffices to show that $\Delta_0$ is injective. Note that $\Delta_0$ is dual to the multiplication map $m_0:M_0\times M_0\to M_0$. Since every singular matrix in $M_0(\U)$ is a product of two singular matrices, $m_0(\U)$ is surjective and, therefore, $\Delta_0$ is injective.
\end{proof} 

A $G$-module $W$ is called \emph{homogeneous} if all its non-zero elements have the same degree. For $d$, $k \Ge 0$, let $P_d^k\subset P$ be the subspace spanned by the differential monomials of degree $d$ and weight $\le k$ (see~\eqref{eq:wtdef}). Note that all $P_d^k$ are $\SL_2$-invariant. We have
$$P_d^0=\Span_\K\left\{x^d, x^{d-1}y,\dots,y^d\right\}\subset P.$$ Let 
\begin{equation}\label{eq:UdWd}
U_d = \Span_\K\left\{P_d^0, \left(x^d\right)', \left(x^{d-1}y\right)',\dots,\left(y^d\right)'\right\}\subset P_d^1\quad\text{and}\quad W_d=P_d^0+(x'y-xy')\cdot P_{d-2}^0\subset P_d^1,
\end{equation} 
which are $\SL_2$-submodules with  
 $U_d$ being isomorphic to $F{\left(P_d^0\right)}$, the prolongation of $P_d^0$, see~\eqref{eq:prolongation}. 
\begin{theorem}\label{thm:main} Let $V$ be a differential representation of $\SL_2$ that is a non-split extension of two irreducible representations $V_1$ and $V_2$ of $\SL_2$, that is, there is a short exact sequence
$$
\begin{CD}
0 @>>> V_1 @>>> V @>>> V_2@>>> 0,
\end{CD}
$$ (hence, $V \in \Rep_0\SL_2$). Then there exists $d \Ge 1$ such that either $V$ or $V^\vee$ is isomorphic to either \begin{enumerate}
\item $U_d$, in which case $\dim V = 2d+2$, or
\item $W_d$, in which case $\dim V = 2d$.
\end{enumerate}
Moreover, $U_d^\vee\cong U_d$ and the $G$-modules $$U_d,\: W_d,\: W_d^\vee, \quad d\Ge 1,$$ form the complete list of pairwise non-isomorphic $G$-modules that are non-trivial extensions of simple modules. 
\end{theorem}
\begin{proof}
The proof will consist of the following steps:
\begin{enumerate}
\item embed either $V$ or $V^\vee$ into $B$ using homogeneity, 
\item embed the result into $P$,
\item show that the result is actually inside $P_d^1$,
\item show that $P_d^1$ has only two submodules with simple socle ($U_d$ and $W_d$) that are non-split extensions of two irreducibles,
\end{enumerate} 
which are contained in Theorem~\ref{thm:embedintoB}, Lemmas~\ref{lem:intoP} and~\ref{lem:firstorder}, and Proposition~\ref{prop:final} that follow.

 The last statement of the theorem can be then shown as follows. Since the $G$-module $P_d^0$ is self-dual~\cite[Theorem 7.2]{HumphreysLie}, we have $$U_d\cong U_d^\vee$$ (see~\cite[Lemma~11]{OvchRecoverGroup}). Note that the simple subquotients of $U_d$ have equal dimensions and the dimensions of simple subquotients of $W_d$ differ by $2$. Hence, $$U_d\not\cong W_s\quad d,s\Ge 1.$$ We also have $$U_d\not\cong U_s,\quad W_d\not\cong W_s,\quad \text{and}\quad W_d\not\cong W_s^\vee,\quad d\neq s\Ge 1,$$ because of the different dimensions. Finally, $$W_d\not\cong W_d^\vee,$$ because  the dimensions of the socles differ.
\end{proof}

\begin{lemma}\label{lem:row}
Let a $G$-module $V$ have a simple socle $U$ and the comodule map $$\rho: V\to V\otimes_\K A, \quad e_j\mapsto \sum_{i=1}^n e_i\otimes a_{ij},\ 1\Le j\Le n,$$ where $\{e_1,\dots,e_n\}$ is a basis of $V$ such that $e_1,\dots,e_k$ form a basis of $U$. Then the elements $a_{1j}$, $1\le j\le n$, form a basis of a submodule $W\subset A$ isomorphic to $V$.   
\end{lemma}

\begin{proof}
Since the $G$-equivariant map $$V \to W \subset A,\quad e_j \mapsto a_{1j},\ 1\Le j \Le n,$$
is non-zero on the socle of $V$ (see \cite[Lemma~3]{OvchRecoverGroup}), it is injective by Proposition~\ref{inj}.
\end{proof}

\begin{theorem}\label{thm:embedintoB} Let $V \subset A$ be a differential representation of $\SL_2$ that is an extension of two irreducible representations of $\SL_2$. Then either $V$ or $V^\vee$ embeds into $B$. 
\end{theorem}
\begin{proof}
We will first show that if $V\subset A$ is a homogeneous $G$-submodule, then $V$ embeds into $B$. Let
$$
\pi_d:C_{\le d}=\bigoplus_{i=0}^d C_d\to C_d
$$
be the projection on the highest-degree component. For the restrictions of $\pi_A$ and $\pi_B$ to submodules $W\subset C$, we will use the same notation, for instance, $\pi_A: W\to A$. We will show that there is a $G$-equivariant morphism $\alpha_d: A_{\Le d}\to B$ making the following diagram of morphisms of $G$-modules commutative:
$$
\begin{CD}
C_{\Le d} @>\pi_d>>C_d\\
@V{\pi_A}VV @VV{\pi_B}V\\
A_{\Le d}@>\alpha_d>>B 
\end{CD}
$$
Equivalently, $$\pi_d(\Ker\pi_A\cap C_{\Le d})\subset \Ker\pi_B\cap C_d.$$ Indeed, let $f\in\Ker\pi_A\cap C_{\Le d}$. Since $\Ker\pi_A$ is generated by $(\det-1)$ and the derivatives $(\det-1)^{(i)}=\det{}^{(i)}$, $i\ge 1$, there are $f_0,\dots,f_k\in C_{\Le d}$ such that
$$f=f_0\cdot (\det-1)+f_1\cdot \det{}'+\cdots+f_k\cdot\det{}^{(k)}.$$
Collecting the terms of degree $d$ in the right-hand side, we obtain
\begin{equation}\label{eq:pid}
\pi_d(f) = g_0\cdot \det+g_1\cdot\det{}'+\cdots+g_k\cdot\det{}^{(k)}-h_0,
\end{equation}
where $g_i=\pi_{d-2}(f_i)\in C_{d-2}$, $h_0\in C_d$ and $$h_0\cdot\det=h_1\cdot\det{}'+\cdots+h_k\cdot\det{}^{(k)},$$ where $h_i=\pi_d(f_i)\in C_d$. Since the differential ideal ${\left[\det{}'\right]}\subset C$ is prime by Lemma~\ref{lem:detprime}  and does not contain $\det$, $$h_0\in{\left[\det{}'\right]}\subset\Ker\pi_B.$$ It follows from \eqref{eq:pid} that $$\pi_d(f)\in \Ker\pi_B\cap C_d,$$ which proves the existence of $\alpha_d$. 
Moreover, we have
$$
\Ker\alpha_d=A_{\Le d-1}.
$$
Indeed, $A_{\Le d-1}\subset\Ker\alpha_d$, because $C_{\Le d-1}=\Ker\pi_d$. On the other hand, if $f\in C_{\Le d}$ and $\pi_d(f)\in\Ker\pi_B$, then $$f=g_0\cdot\det+g_1\cdot\det{}'+\cdots+g_k\cdot\det{}^{(k)}+g$$ for some $g_i\in C_{d-2}$, $0\Le i \Le k$, and $g\in C_{\Le d-1}$. Hence, $f$ is congruent to $g_0+g\in C_{\Le d-1}$ modulo $\Ker\pi_A$, and $\pi_A(f)\in A_{\Le d-1}$.

We conclude that $V$ embeds into $B$ via $\alpha_d$ for some $d$ if (and only if) $V$ is homogeneous. We want to show that $V$ or $V^\vee$ is homogeneous and, thus, embeds into $B$. We will use the following observation. We say that a $G$-module $W$ has  degree $d$, and write $\deg W = d$, if the image of the comodule map $$\rho: W\to W\otimes_\K A$$ lies in $W\otimes_\K A_{\Le d}$ and not in $W\otimes_\K A_{\Le d-1}$. We will show that if $U\subset W$ is a $G$-submodule, then
\begin{equation}\label{eq:max}
\deg W = \max\{\deg U, \deg W/U\}.
\end{equation}
Set $$d=\deg W,\quad d_1=\deg U,\quad \text{and}\quad d_2=\deg W/U.$$ We have $$d\Ge \max\{d_1, d_2\}.$$ Fix a $\K$-basis $\{w_1,\dots,w_n\}$ of $W$ such that $w_1,\dots,w_k$ form a basis of $U$, and define $a_{ij}\in A$, $1\Le i,j\Le n$, by
$$
\rho(w_j)=\sum_{i=1}^nw_i\otimes a_{ij}.
$$
Then, for the comultiplication $\Delta: A\to A\otimes_\K A$, we have~\cite[Corollary~3.2]{Waterhouse}
\begin{equation}\label{eq:comod}
\Delta(a_{ij})=\sum_{l=1}^na_{il}\otimes a_{lj}=\sum_{l=1}^ka_{il}\otimes a_{lj}+\sum_{l=k+1}^na_{il}\otimes a_{lj}.
\end{equation}
Let $i$ and $j$, $1\Le i,j\Le n$, be such that $\deg a_{ij}=d$. Then, by Lemma~\ref{lem:degree}, the left-hand side of~\eqref{eq:comod} does not belong to $$A_{\Le d-1}\otimes_\K A_{\Le d}+A_{\Le d}\otimes_\K A_{\Le d-1},$$ while the right-hand side of~\eqref{eq:comod} belongs to $$A_{\Le d_1}\otimes_\K A_{\Le d}+A_{\Le d}\otimes_\K A_{\Le d_2}.$$ This is possible only if
$$
d\Le \max\{d_1, d_2\},
$$
which proves~\eqref{eq:max}.

For any $W\in\Rep G$, we have $\deg W=\deg W^\vee$. Indeed, if $\{a_{ij}\}$ are the matrix entries corresponding to $W$ in some basis, then, in the dual basis, the entries form the set $\{S(a_{ij})\}$, where $S:A\to A$ is the antipode. Since $S$ does not increase the degree (this is seen from its action on the generators $x_{ij}\in A$) and $S^2=\Id$, $S$ preserves the degree.  

If a submodule $W\subset A$ is not homogeneous, $W$ contains a proper submodule of smaller degree. Suppose $V$ is not homogeneous. Then, for the socle $U\subset V$, we have $\deg U<\deg V$. Then, by~\eqref{eq:max}, $$\deg V=\deg V/U.$$ Since $V/U$ is simple, so is $(V/U)^\vee\subset V^\vee$. We have $$\deg V^\vee=\deg(V/U)^\vee,$$ and, therefore, $V^\vee$ does not contain a proper submodule of smaller degree. Hence, $V^\vee$ is homogeneous.
\end{proof}

\begin{example}
Set $x_{ij}:=\pi_A(c_{ij})$, see~\eqref{eq:ABP}. Let $$V = \Span_\K{\left\{1,x_{11}'x_{21}-x_{11}x_{21}',x_{12}'x_{22}-x_{12}x_{22}',x_{11}'x_{22}-x_{21}'x_{12}\right\}} \subset A,$$ which is an $\SL_2$-submodule but not homogeneous, and, hence, the map $V \to B$ defined in the proof of Theorem~\ref{thm:embedintoB} is not injective on $V$. However, $$V^\vee \cong \Span_\K{\left\{x_{11}^2,x_{11}x_{12},x_{12}^2,x_{11}'x_{12}-x_{11}x_{12}'\right\}} 
\subset A$$ is homogeneous and, therefore, embeds into $B$. \end{example}

\begin{lemma}\label{lem:intoP} Let $V \subset B$ have simple socle (see Section~\ref{sec:ss}). Then $V$ embeds into $\K\{x,y\}$.
\end{lemma}
\begin{proof}
By Proposition~\ref{subquot}, $V_{soc}$ is algebraic. Hence, by Lemma~\ref{lem:row}, $$V\simeq W\subset A\quad \text{with}\quad W_{soc}\subset\K[c_{ij}]/(\det-1).$$ Moreover, since $B$ is the direct sum of $\pi_B(C_d)$, $d\ge 0$, $V$ is homogeneous and, therefore, so is $W$. As in the proof of Theorem~\ref{thm:embedintoB}, $W$ embeds into $B$ so that its image $\tilde V\cong V$ has the socle in the non-differential polynomials $\K[x,y,x_1,y_1]$, where 
$$
x:=\pi_B(c_{11}),\quad y:=\pi_B(c_{12}),\quad x_1:=\pi_B(c_{21}),\quad y_1:=\pi_B(c_{22}).
$$
Therefore, without loss of generality, we may assume $V_{soc}\subset\K[x,y,x_1,y_1]$.

Let $0 \ne f \in  V_{soc}$. Since $\U$ is algebraically closed, there exists $0\neq (a,b,a_1,b_1) \in \U^4$ such that $$ab_1-ba_1 = 0,\quad \text{and}\quad f(a,b,a_1,b_1) \ne 0.$$
Suppose $a\ne 0$ (the cases $b\ne 0$, $a_1\ne 0$, $b_1\ne 0$ are considered similarly). Set $\alpha := a_1/a$. Then $b_1 = b\alpha$. So,
$f(a,b,a\alpha,b\alpha) \ne 0$, which implies that 
$$
0 \ne g_f(x,y,z) := f(x,y,xz,yz) \in \K\{x,y,z\}.
$$
Since $\Q \subset\K$ is infinite and the polynomial $g_f$ is non-zero, there exists $\beta \in \K$ such that $g(x,y,\beta) \ne 0$. Therefore, the $\SL_2$-equivariant differential ring homomorphism
$$
\varphi : B \to \K\{x,y\},\quad h(x,y,x_1,y_1) \mapsto h(x,y,x\beta,y\beta)
$$
is injective on $V_{soc}$. Thus, by Proposition~\ref{inj}, $\varphi$ is injective on $V$. 
\end{proof}

\begin{lemma}\label{lem:firstorder} Let $V \subset \K\{x,y\}$ be from $\Rep_0G$, where $G = \SL_2$, and be an extension of two irreducible $G$-modules. Then $V\subset P_d^1$ and $V_{soc}= P_d^0$.
\end{lemma} 
\begin{proof}
The torus $\Gm\subset\SL_2$ (embedded as diagonal matrices with entries $a$ and $a^{-1}$) acts on $P$. By the representation theory of $\SL_2$, $P_d^0$ is simple. The rest follows from Lemma~\ref{lem:free} and the observation that $V\cap P_d^k$ is a submodule of $V$ for every $k\Ge 0$.
\end{proof}

\begin{proposition}\label{prop:final}
Let $V\subset P_d^1$ be a submodule that is an extension of two irreducible $G$-modules.  
Then  $V=U_d$ or $V=W_d$, see \eqref{eq:UdWd}.
\end{proposition}
\begin{proof}
Note that $$W_d/P_d^0\simeq P_{d-2}^0\quad \text{and}\quad U_d/P_d^0\simeq P_d^0.$$ Hence, for the quotient map $$q: P_d^1\to P_d^1/P_d^0,$$ the sum $q(U_d)+q(W_d)$ is direct.
Since $$\dim \im q = 2d = \dim P_d^0 + \dim P_{d-2}^0,$$ 
we have $$\im q = q(U_d)\oplus q(W_d).$$ Since $V\supset P_d^0$ by Lemma~\ref{lem:firstorder}, $q(V)$ is irreducible and, therefore, must coincide with one of the summands. Finally, if, for instance, $q(V)=q(U_d)$, then $$V=q^{-1}(q(V))=q^{-1}(q(U_d))=U_d,$$ because $\Ker q \subset V$ and $\Ker q\subset U_d$. Similarly, if $q(V)=q(W_d)$, then $V=W_d$.
\end{proof}

\subsection{Example}
If we omit the requirement for $V$ being an extension of {\it two} irreducible $\SL_2$-modules, then the claim of Theorem~\ref{thm:embedintoB} is no longer true as the following example shows. 
\begin{example}\label{ex:counterexample} 
Let $V =\Span_\K\left\{1, x_{11}'x_{21}-x_{11}x_{21}',x_{11}'x_{22}-x_{21}'x_{12},x_{12}'x_{22}-x_{12}x_{22}',x_{11}'x_{22}-x_{12}'x_{21}\right\} \subset A$, 
which gives
the following differential representation of $\SL_2$:
$$
\SL_2(\U) \ni \begin{pmatrix}
a& b\\
c& d
\end{pmatrix} \mapsto
\begin{pmatrix}
1 & a'c-ac'& a'd-bc'& b'd-bd'&a'd'-b'c'\\
0&a^2 & ab & b^2& ab'-a'b \\
0&2ac&ad+bc& 2bd& 2(ad'-bc')\\
0&c^2& cd& d^2&cd'-c'd\\
0&0&0&0&1
\end{pmatrix},
$$
with neither $V$ nor $V^\vee$ embeddable into $\K\{x,y\}$. Indeed, if $V$ were embeddable into $\K\{x,y\}$, then 
the submodule $$W := \Span_\K\left\{1, x_{11}'x_{21}-x_{11}x_{21}',x_{11}'x_{22}-x_{21}'x_{12},x_{12}'x_{22}-x_{12}x_{22}'\right\}$$ would have an embedding into $\K\{x,y\}$. We will show that this is impossible. For this, first let $\varphi(W) \subset \K\{x,y\}$. Note that a calculation shows that $W\cong U^\vee$ for $$U := \Span_\K\left\{x^2,xy,y^2,x'y-xy'\right\} \subset \K\{x,y\},$$ which was discussed in Example~\ref{ex:reg}. Since the only $\SL_2$-invariant elements of $\K\{x,y\}$ are $\K$, there would be a splitting of $$\varphi(W)\cong\K\oplus\ker\psi,\quad \psi : \varphi(W) \to \K,\ \ w \mapsto w(0),$$ that is, taking the term with no $x$ and $y$ in it. However, this would mean that $W^\vee \cong U$ splits into a direct sum of two modules of dimension $3$ and $1$ as well, that is,
$$
\begin{CD}
0 @>>> \Span_\K\left\{x^2,xy,y^2\right\} @>>> U @>\pi>> \K @>>> 0,
\end{CD}
$$
where $\pi$ is the usual quotient map,
has a splitting $s : \K \to U$ such that $\pi\circ s = \id_\K$. 

On the one hand,
$\{0\} \ne s(\K) \subset U$ is $\SL_2$-invariant, because $\K$ is and $s$ is $\SL_2$-equivariant. On the other hand, the only $\SL_2$-invariant element in $U$ is $0 \in U\cap\K$. This is a contradiction, implying that $V$ does not embed into $\K\{x,y\}$. Since the diagonal blocks of $V^\vee$ have the same dimensions and are in the same order, $(1,3,1)$, as $V$, the above
argument also shows that $V^\vee$ does not embed into $\K\{x,y\}$.
\end{example}

\section{Acknowledgments} We are grateful to Phyllis Cassidy, Sergey Gorchinskiy, Charlotte Hardouin, Ray Hoobler, Marina Kondratieva, Nicole Lemire, Mircea Musta\c{t}\u{a}, Lex Renner, Michael Singer, William Sit,  Dmitry Trushin, Alexey Zobnin, and the referee for their helpful comments.

\bibliographystyle{spmpsci}
\bibliography{sl2}

\end{document}